\documentclass[a4paper,12pt]{article}
\usepackage{times, url}
\usepackage{graphicx,amsmath,amssymb,amsthm,amsfonts,bm,float,cite}
\newtheorem{theorem}{Theorem}[section]
\newtheorem{corollary}{Corollary}[section]

\newtheorem{proposition}{Proposition}[section]

\newtheorem{example}{Example}[section]
\newtheorem{remark}{Remark}[section]

\title{On some classes of binary matrices}
\author{Krasimir Yordzhev}
\date{}

\begin{document}

\maketitle

Trakia University, Stara Zagora, Yambol, Bulgaria

Email address: krasimir.yordzhev@gmail.com

\begin{abstract} The work considers the set $\Lambda_n^k$ of all $n\times n$ binary matrices having the same number of $k$ units in each row and each column. The article specifically focuses on the matrices whose rows and columns are sorted lexicographically. We examine some particular cases and special properties of this matrices. Finally, we demonstrate the relationship between the Fibonacci numbers and the cardinality of two classes of $\Lambda_n^k$-matrices with lexicographically sorted rows and columns. 
\end{abstract}

{\bf Keywords:} Binary matrix, sorted lexicographically, decreasing order, increasing order, Fibonacci numbers. 

{\bf 2020 Mathematics Subject Classification:} 15B34, 11B39. 
\vspace{5mm}

\section{Preliminaries and notations} \label{sec:Intr}

 A \emph{binary} (or   \emph{boolean}, or (0,1)-\emph{matrix}) is a matrix whose all elements belong to the set $\mathfrak{B} =\{ 0,1 \}$. 
Let $m$  and $n$  be positive integers. With $\mathfrak{B}_{n\times m}$  we will denote the set of all $n\times m$  binary matrices and with  $\mathfrak{B}_n$  we will denote the set of all binary $n$-vectors.

If $n$ and $k$ are integers such that  $n\ge 2$, $0\le k\le n$, then we will call \emph{$\Lambda_n^k$-matrices} all $n\times n$ binary matrices in each row and each column of which there are exactly $k$ unity elements. 

The set $$\Lambda_n^k \subset {\mathfrak B}_{n\times n}$$ is the set of all $\Lambda_n^k$-matrices.

Let $A=(a_{ij} )\in \mathfrak{B}_{n\times m}$. With $r(A)$ we will denote the ordered  $n$-tuple
$$r(A) =\langle x_1, x_2 , \ldots ,x_n \rangle ,$$ 
where  $0\leq x_i \le 2^m -1$,   $i=1,2,,\ldots n$ and $x_i$  is the integer written in binary notation with the help of the $i$-th row of  $A$, i.e
$$x_i = \sum_{j=1}^{m} a_{ij} 2^{m-j} ,\quad i=1,2,,\ldots n .$$

Similarly with $c(A)$ we will denote the ordered $m$-tuple
$$c(A)=\langle y_{1} ,y_{2} ,\ldots ,y_{m} \rangle ,$$
where $0\le y_{j} \le 2^{n} -1$, $j=1,2,\ldots m$ and $y_{j} $ is a natural number written in binary notation with the help of the $j$-th column of $A$ i.e
$$y_j = \sum_{i=1}^{n} a_{ij} 2^{n-i} ,\quad j=1,2,,\ldots m .$$

 Let $A\in {\mathfrak B}_{n\times m}$,
$r(A)=\langle x_{1} ,x_{2} ,\ldots ,x_{n} \rangle$ and $c(A)=\langle y_{1} ,y_{2} ,\ldots ,y_{m} \rangle$. Then with $\mathfrak{C}_{n\times m}$ and with  $\mathfrak{D}_{n\times m}$  we will denote the sets: 
$$\mathfrak{C}_{n\times m} = \left\{ A\in \mathfrak{B}_{n\times m} \  \left| \  x_{1} \le x_{2} \le \cdots \le x_{n} \quad \mathrm{and} \quad y_{1} \le y_{2} \le \cdots \le y_{m} \right. \right\} \subset \mathfrak{B}_{n\times m} ,$$
$$\mathfrak{D}_{n\times m} = \left\{ A\in \mathfrak{B}_{n\times m} \  \left| \  x_{1} \ge x_{2} \ge \cdots \ge x_{n} \quad \mathrm{and} \quad y_{1} \ge y_{2} \ge \cdots \ge y_{m} \right. \right\} \subset \mathfrak{B}_{n\times m} .$$

In other words, $A\in \mathfrak{C}_{n\times m}$ if and only if  rows and columns of $A$ are sorted in lexicographical non decreasing order and $A\in \mathfrak{D}_{n\times m}$ if and only if  rows and columns of $A$ are sorted in lexicographical non increasing order.

We define the sets
$$\Gamma_n^k =\mathfrak{C}_{n\times n} \cap \Lambda_n^k ,$$
$$\Delta_n^k =\mathfrak{D}_{n\times n} \cap \Lambda_n^k$$
and the functions 
$$\gamma (n,k) =\left| \Gamma_n^k \right| ,$$
$$\delta (n,k) =\left| \Delta_n^k \right| .$$

As is well known (see for example\cite{atanassov} or \cite{koshy}), the sequence $\displaystyle \left\{ f_n \right\}_{n=0}^\infty $ of \emph{Fibonacci numbers} is defined by the recurrence relation
$$f_0 =f_1 =1, \qquad f_n = f_{n-1} +f_{n-2} \quad \textrm{for} \quad n=2,3,\ldots$$

\section{Some properties of the sets $\Gamma_n^k$ and $\Delta_n^k $.}

\begin{proposition}\label{ThThrm}
In general, 
$$
  \gamma (n,k) \ne \delta  (n,k)
$$
\end{proposition}

\begin{proof}

According to \cite[Sequence A229162]{OEIS} and publication \cite{yordzhev3013}, some values of the integer sequence, obtained using a computer program are $$\displaystyle \left\{ \gamma (n,k)\right\}_{n=1}^{10} = A229162 =\left\{ 0, 0, 1, 1, 3, 25, 272, 4070, 79221, 1906501\right\} .$$ 

It is not difficult to see that all $\Gamma_5^3$-matrices are as follows:

$$\left(
          \begin{array}{ccccc}
            0 & 0 & 1 & 1 & 1 \\
            0 & 0 & 1 & 1 & 1 \\
            1 & 1 & 0 & 0 & 1 \\
            1 & 1 & 0 & 1 & 0 \\
            1 & 1 & 1 & 0 & 0 \\
          \end{array}
        \right),\qquad 
        \left(
          \begin{array}{ccccc}
            0 & 0 & 1 & 1 & 1 \\
            0 & 1 & 0 & 1 & 1 \\
            1 & 0 & 0 & 1 & 1 \\
            1 & 1 & 1 & 0 & 0 \\
            1 & 1 & 1 & 0 & 0 \\
          \end{array}
         \right) \quad \textrm{and} \quad
         \left(
          \begin{array}{ccccc}
            0 & 0 & 1 & 1 & 1 \\
            0 & 1 & 0 & 1 & 1 \\
            1 & 0 & 1 & 0 & 1 \\
            1 & 1 & 0 & 1 & 0 \\
            1 & 1 & 1 & 0 & 0 \\
          \end{array}
        \right)
 $$

According to \cite[Sequence A181344]{OEIS}  all $\Delta_5^3$-matrices are as follows:

$$ \left(
          \begin{array}{ccccc}
            1 & 1 & 1 & 0 & 0 \\
            1 & 1 & 0 & 1 & 0 \\
            1 & 1 & 0 & 0 & 1 \\
            0 & 0 & 1 & 1 & 1 \\
            0 & 0 & 1 & 1 & 1 \\
          \end{array}
        \right) ,\qquad 
         \left(
          \begin{array}{ccccc}
            1 & 1 & 1 & 0 & 0 \\
            1 & 1 & 1 & 0 & 0 \\
            1 & 0 & 0 & 1 & 1 \\
            0 & 1 & 0 & 1 & 1 \\
            0 & 0 & 1 & 1 & 1 \\
          \end{array}
         \right) , \qquad 
         \left(
          \begin{array}{ccccc}
            1 & 1 & 1 & 0 & 0 \\
            1 & 1 & 0 & 1 & 0 \\
            1 & 0 & 1 & 0 & 1 \\
            0 & 1 & 0 & 1 & 1 \\
            0 & 0 & 1 & 1 & 1 \\
          \end{array}
        \right) ,
 $$
 $$
         \left(
          \begin{array}{ccccc}
            1 & 1 & 1 & 0 & 0 \\
            1 & 1 & 0 & 1 & 0 \\
            1 & 0 & 0 & 1 & 1 \\
            0 & 1 & 1 & 0 & 1 \\
            0 & 0 & 1 & 1 & 1 \\
          \end{array}
         \right) \quad \textrm{and} \quad
         \left(
          \begin{array}{ccccc}
            1 & 1 & 1 & 0 & 0 \\
            1 & 0 & 0 & 1 & 1 \\
            1 & 0 & 0 & 1 & 1 \\
            0 & 1 & 1 & 1 & 0 \\
            0 & 1 & 1 & 0 & 1 \\
          \end{array}
        \right)
 $$

So $\gamma (5,3)=3$ and $\delta  (5,3)=5$, which proves Proposition \ref{ThThrm}.
\end{proof}

\begin{proposition}\label{Proposition1}
Let $A=(a_{ij} )\in {\mathfrak C}_{n\times m} $. Then there exist integers $s,t$, such that $1\le s\le m$, $1\le t\le n$ and
\begin{equation} \label{zGrindEQ__2_}
 a_{1 1} =a_{1 2} =\cdots =a_{1 s} =0,\quad a_{1, s+1} =a_{1, s+2} =\cdots =a_{1 m} =1, \end{equation}
\begin{equation} \label{zGrindEQ__3_}
a_{1 1} =a_{2 1} =\cdots =a_{t 1} =0,\quad a_{t+1, 1} =a_{t+2, 1} =\cdots =a_{n 1} =1. \end{equation}
\end{proposition}

\begin{proof}
Let $r(A)=\langle x_{1} ,x_{2} ,\ldots x_{n} \rangle $ and $c(A)=\langle y_{1} ,y_{2} ,\ldots y_{m} \rangle $. We assume that there exist integers $p$ and $q$, such that $1\le p<q\le m$, $a_{1 p} =1$ and $a_{1 q} =0$. In this case, $y_{p} > y_{q}$, which contradicts the condition that columns of $A$ are sorted in lexicographical non decreasing order. We have proven (\ref{zGrindEQ__2_}). Similarly, we prove (\ref{zGrindEQ__3_}) as well.

\end{proof}

\begin{corollary} \label{Corollary1}
Let $A=(a_{ij} )\in {\mathfrak C}_{n\times m} $. Then there exist integers $s,t$,  $0\le s\le m$, $0\le t\le n$, such that 
$$x_{1} =2^{s} -1 $$
and
$$ y_{1} =2^{t} -1 ,$$
where $s$ equals the number of units in the first row and $t$ equals the number of units in the first column of $A$.

\hfill $\square$
\end{corollary}

\begin{proposition}\label{Proposition2} (Dual of Proposition \ref{Proposition1}) 
Let $A=(a_{ij} )\in {\mathfrak D}_{n\times m} $. Then there exist integers $s,t$, such that $1\le s\le m$, $1\le t\le n$ and
\begin{equation} \label{yyGrindEQ__2_}
 a_{1 1} =a_{1 2} =\cdots =a_{1 s} =1,\quad a_{1, s+1} =a_{1, s+2} =\cdots =a_{1 m} =0, \end{equation}
\begin{equation} \label{yyGrindEQ__3_}
a_{1 1} =a_{2 1} =\cdots =a_{t 1} =1,\quad a_{t+1, 1} =a_{t+2, 1} =\cdots =a_{n 1} =0. \end{equation}

\hfill $\square$
\end{proposition}

\begin{corollary} \label{Coroll2 } (Dual of Corollary \ref{Corollary1}) 
Let $A=(a_{ij} )\in {\mathfrak D}_{n\times m} $. Then there exist integers $s,t$,  $0\le s\le m$, $0\le t\le n$, such that 
$$x_{1} =(2^{s} -1)2^{m-s} = 2^m -2^{m-s}$$ 
and
$$y_{1} =(2^{t} -1)2^{n-t} = 2^n -2^{n-t} ,$$
where $s$ equals the number of units in the first row and $t$ equals the number of units in the first column of $A$.

\hfill $\square$
\end{corollary}

\begin{theorem} \label{ttrgamma_njk}
Let $n$ and $k$, be integers such that $n\ge 1$, $0\le k$ and $k\le n$. Then
$$\gamma (n,n-k) = \delta  (n,k) .$$
\end{theorem}

\begin{proof}

Let $a\in \mathfrak{B} =\{ 0,1 \}$. Then with $\overline{a}$ we will denote 
$$ \overline{a} = \left\{ \begin{array}{ccc}
                          1 & \textrm{if} & a=0; \\
                          0 & \textrm{if} & a=1. 
                        \end{array}
                        \right.
 $$

Obviously $\overline{\overline{a}} =a$.

If $u=\langle u_1 ,u_2 ,\ldots ,u_n \rangle \in\mathfrak{B}_n$ then with $\overline{u}$ we will denote $\overline{u} =\langle \overline{u}_1 ,\overline{u}_2 ,\ldots ,\overline{u}_n \rangle $. If $A=(a_{ij} )\in {\mathfrak B}_{n\times m} $ then with $\overline{A}$ we will denote $\overline{A}=(\overline{a}_{ij} )$.

Let  $u =\langle u_{1} ,u_{2} ,\ldots ,u_{n} \rangle $, $v =\langle v_{1} ,v_{2} ,\ldots ,v_{n} \rangle \in \mathfrak{B}_n$. Then it is easy to see that $u < v$ if and only if  $\overline{u} > \overline{v}$.
Therefore a matrix  $A=(a_{ij} )\in \mathfrak{C}_{n\times m} $  if and only if the matrix  $\overline{A}=(\overline{a}_{ij} )\in \mathfrak{D}_{n\times m} $.

Finally, we take into account the fact that the matrix $A=\left( a_{ij}\right) \in \Lambda_n^{n-k}$  if and only if the matrix $\overline{A}=\left( \overline{a}_{ij}\right) \in \Lambda_n^{k}$.

\end{proof}

\begin{theorem}\label{partitofM}
Let $n$ be an integer, $n\ge2$ and let $\displaystyle A=\left( a_{ij} \right) \in \Delta_n^2 \subset \Lambda_n^2$. Then $A$ has the form:
\begin{equation}\label{onestarA}
A=
\left(
  \begin{array}{cccccccc}
      &   &   &   & 0 & 0 & \cdots & 0 \\
      & B &   &   & 0 & 0 & \cdots & 0 \\
      &   &   &   & \vdots & \vdots & \ddots  & \vdots \\
      &   &   &   & 0 & 0 & \cdots & 0 \\
    0 & 0 & \cdots & 0 &   &   &   &   \\
    0 & 0 & \cdots & 0 &   & C  &   &   \\
    \vdots & \vdots & \ddots  & \vdots &   &  &   &   \\
    0 & 0 & \cdots & 0 &   &   &   &   \\
  \end{array}
\right) ,  
\end{equation}
where $B$ and $C$ are square binary matrices,
\begin{equation}\label{B-zero}
B= \left(
     \begin{array}{cc}
       1 & 1 \\
       1 & 1 \\
     \end{array}
   \right) ,
\end{equation}
or $B$  has the form:
\begin{equation}\label{twostarts2}
B=
\left(
  \begin{array}{cccccccc}
    1 & 1 & 0 & 0 & 0 & \cdots & 0 & 0 \\
    1 & 0 & 1 & 0 & 0 & \cdots & 0 & 0 \\
    0 & 1 & 0 & 1 & 0 & \cdots & 0 & 0 \\
    0 & 0 & 1 & 0 & 1 &\cdots & 0 & 0 \\
    \vdots & \vdots & \vdots & \ddots & \ddots & \ddots & \vdots & \vdots \\
    0 & 0 & 0 & \cdots & 1 & 0 & 1 & 0 \\
    0 & 0 & 0 & \cdots & 0 & 1 & 0 & 1 \\
    0 & 0 & 0 & \cdots & 0 & 0 & 1 & 1 \\
  \end{array}
\right) .
\end{equation}
$C\in \Delta_s^2 \subset \Lambda_s^2$ for some $s$ such that $2\le s \le n-2$, or $C$ does not exist. The remaining elements of matrix $A$, which are outside submatrices $B$ and $C$, are equal to 0.

\end{theorem}

\begin{proof}
Such that $A\in \Lambda_n^2$ and from Proposition \ref{Proposition2} it follows that $a_{11} =a_{12} =a_{21} =1$, $a_{i1} =0$  and $a_{1j} =0$ for $3\le i,j\le n$.

i) If $a_{22} =1$ then $B$ has the form (\ref{B-zero}). If the matrix $C$ exists, then it is easy to see that $C\in \Delta_{n-2}^{2} \subset \Lambda_{n-2}^{2}$ matrix.

ii) Let $a_{22} =0$, i.e. $A$ be of the form $$ \left( \begin{array}{ccccc}
                                                                1 & 1 & 0 & \cdots & 0 \\
                                                                1 & 0 & a_{23} &        &  \\
                                                                0 & a_{32} & &  &  \\
                                                                \vdots & & &  &  \\
                                                                0 & & &  &  \\
                                                              \end{array}
                                                            \right) .$$ 

Let $c(A)=\langle y_{1} ,y_{2} ,\ldots ,y_{m} \rangle$. We suppose that $a_{23} =0$.  Since $A\in \Lambda_n^{2}$ there exists an integer $t$ such that $3<t\le n$ and $a_{2t} =1$. In this case, it is easy to see that $y_3 <y_t$, which is impossible because $A\in {\mathfrak D}_{n\times n} $. Therefore $a_{23}=1$. Similarly $a_{32}=1$. Therefore, when $a_{22}=0$, $A$ is represented as $A= \left(
                                                                                                                                          \begin{array}{cccc}
                                                                                                                                            1 & 1 & 0 & \cdots \\
                                                                                                                                            1 & 0 & 1 &        \\
                                                                                                                                            0 & 1 & a_{33} &   \\
                                                                                                                                            \vdots &  &    &   \\   
                                                                                                                                          \end{array}
                                                                                                                                        \right)$.
We consider again the two possible cases for $a_{33}$: $a_{33} =1$ or $a_{33} =0$. When $a_{33} =1$, the statement is proved. When $a_{33} =0$, we do the same reasoning as above. This process cannot continue indefinitely, since $n$ is a finite integer. Therefore, there exists an integer $t$, $2\le t\le n$ such that $a_{t\, t} =a_{t-1\, t} =a_{t\, t-1} =1$, i.e. in the upper left corner of A there is a submatrix of the form (\ref{twostarts2}). And in this case, it is easy to see that if the matrix $C$ exists, then $C\in \Delta_s^2 \subset \Lambda_s^{2}$ for some $s$ such that $2\le s \le n-2$.

\end{proof}

\begin{corollary} \label{remrkbss}
Let $n$ be an integer, $n\ge 2$.Then 
$$\delta (n,2)=\gamma (n,n-2)=\emph{number of all ordered $s$-tuples of integers} $$
$$\langle p_1 ,p_2 ,\ldots ,p_s \rangle , \quad 1\le s\le \left[ \frac{n}{2} \right],$$
such that $2\le p_i \le n$, $i=1,2,\ldots s$ and 
$$p_1 +p_2 +\cdots +p_s =n .$$

\end{corollary}

\begin{remark} 
Corollary \ref{remrkbss} can be used as a basis for the proof of well known equation formulated and proven in \cite{6}. 
$$\left| \Lambda_n^2 \right| =\left| \Lambda_n^{n-2} \right| = \sum_{2p_2 +3p_3 + \cdots +np_n =n} \frac{(n!)^2}{\displaystyle \prod_{r=2}^n p_r !(2r)^{x_r}} .$$
\end{remark}

\begin{remark}
A similar theorem can be formulated and proved for the set  $\Gamma_{n}^{n-2} \subset \Lambda_{n}^{n-2}$, $n\ge 2$. 
\end{remark}

\begin{example}\label{ExampleEx3.1}
\ 

i) There is only one $\Delta_2^{2}$ matrix: $$\left(
                                                                                              \begin{array}{cc}
                                                                                                1 & 1 \\
                                                                                                1 & 1 \\
                                                                                              \end{array}
                                                                                            \right) .$$

ii) There is only one $\Delta_3^{2}$ matrix: $$\left(
                                                                                             \begin{array}{ccc}
                                                                                               1 & 1 & 0 \\
                                                                                               1 & 0 & 1 \\
                                                                                               0 & 1 & 1 \\
                                                                                             \end{array}
                                                                                           \right) .$$

iii) There are two $\Delta_4^{2}$matrices: $$\left(
                                                                                          \begin{array}{cccc}
                                                                                            1 & 1 & 0 & 0 \\
                                                                                            1 & 1 & 0 & 0 \\
                                                                                            0 & 0 & 1 & 1 \\
                                                                                            0 & 0 & 1 & 1 \\
                                                                                          \end{array}
                                                                                        \right) \textrm{and} \left(
                                                                                          \begin{array}{cccc}
                                                                                            1 & 1 & 0 & 0 \\
                                                                                            1 & 0 & 1 & 0 \\
                                                                                            0 & 1 & 0 & 1 \\
                                                                                            0 & 0 & 1 & 1 \\
                                                                                          \end{array}
                                                                                        \right) .$$
                                                                                        
iv)  There are three $\Delta_5^2$ matrices:

$$ \left(
          \begin{array}{ccccc}
            1 & 1 & 0 & 0 & 0 \\
            1 & 1 & 0 & 0 & 0 \\
            0 & 0 & 1 & 1 & 0 \\
            0 & 0 & 1 & 0 & 1 \\
            0 & 0 & 0 & 1 & 1 \\
          \end{array}
        \right),\quad 
         \left(
          \begin{array}{ccccc}
            1 & 1 & 0 & 0 & 0 \\
            1 & 0 & 1 & 0 & 0 \\
            0 & 1 & 1 & 0 & 0 \\
            0 & 0 & 0 & 1 & 1 \\
            0 & 0 & 0 & 1 & 1 \\
          \end{array}
         \right) \quad \textrm{and} \quad
         \left(
          \begin{array}{ccccc}
            1 & 1 & 0 & 0 & 0 \\
            1 & 0 & 1 & 0 & 0 \\
            0 & 1 & 0 & 1 & 0 \\
            0 & 0 & 1 & 0 & 1 \\
            0 & 0 & 0 & 1 & 1 \\
          \end{array}
        \right)
 $$

\end{example}

\section{$\Delta_n^k$ and $\Gamma_n^k$ matrices in relation to the  Fibonacci numbers}

\begin{theorem} \label{thtiorDelta_n_2}
Let $n$ be a nonnegative integer. Then
\begin{equation}\label{Fobdelta}
f_n =\delta (n+2,2) ,
\end{equation}
where $f_n$ is the $n$-th element of the Fibonacci sequence.
\end{theorem}

\begin{proof}
For $n=0,\ 1,\ 2$ and $3$, see Example \ref{ExampleEx3.1}.

Let $n$ be an integer, $n\ge 2$ and let $A= \left( a_{ij} \right)\in\Delta_{n+2}^2$. From Theorem \ref{partitofM} it follows that $A$ is presented in the form (\ref{onestarA}) and the set  $\Delta_{n+2}^2$ is a partition into two disjoint subsets $\mathcal{M}_1$ and $\mathcal{M}_2$ such that the set $\mathcal{M}_1$ consists of matrices $A$ whose upper left corner is a submatrix $B$ of the type (\ref{B-zero}) and the set $\mathcal{M}_2$ consists of matrices $A$ whose upper left corner is a submatrix $B$ of the type (\ref{twostarts2}).
$$\mathcal{M}_1 \cap \mathcal{M}_2 =\emptyset ,\qquad \mathcal{M}_1 \cup \mathcal{M}_2 =\Delta_{n+2}^2 . $$

Therefore 
\begin{equation}\label{partitionM1andM1}
\left| \Delta_{n+2}^2  \right| = \left| \mathcal{M}_1 \right| + \left| \mathcal{M}_2 \right| .
\end{equation}

i) Let $A\in \mathcal{M}_1$. In $A$, we remove the first and second rows and the first and second columns. We obtain the matrix $C\in \Lambda_n^{2}$. From Theorem \ref{partitofM} it follows that  $C\in\Delta_n^2$. 

Conversely, let $C =\left( c_{i\, j} \right) \in \Delta_n^{2}$, $n\ge 2$. From $C$ we obtain the matrix $A=\left( a_{i\, j} \right) \in\Lambda_{n+2}^2$ as follows: $a_{1\, 1} =a_{1\, 2} =a_{2\, 1}=a_{2\, 2} =1$, $a_{1\, j} =a_{2\, j} =0$ for $3\le j\le n+2$ and $a_{i\, 1} =a_{i\, 2} =0$ for $3\le i\le n+2$. For each  $i,j\in \{3,4,\ldots ,n+2 \}$ we assume $a_{i\, j} =c_{i-2\, j-2}$. It is easy to see that the so obtained matrix $A\in \Delta_{n+2}^2$.

Therefore, 
\begin{equation}\label{M1=mu(n+2,n)}
\left| \mathcal{M}_1 \right| =\delta (n,2)
\end{equation}
for any integer $n\ge 2$.

ii) Let $A\in \mathcal{M}_2$, i.e. $A\in \Delta_{n+2}^2$ is of the form 
$$A= \left(
\begin{array}{cccccc}
1 & 1 & 0 & 0 & \cdots & 0 \\
1 & 0 & 1 & 0 & \cdots & 0 \\
0 & 1 &   &   &    &   \\
0 & 0 &   &   &    &   \\
\vdots & \vdots &   &  &  &   \\
0 & 0 &   &   &  &
\end{array}
\right) .
$$ 
We change  $a_{2\, 2}$ from 0 to 1 and remove the first row and the first column of $A$. In this way we obtain a matrix, which is easy to see that it belongs to the set $\Delta_{n+1}^2$.

Conversely, let $D=\left( d_{i\, j} \right) \in \Delta_{n+1}^{2}$. According to Proposition \ref{Proposition1} $d_{1\, 1} =d_{1\, 2} =d_{2\, 1} =1$. We change  $d_{1\, 1}$  from 1 to 0. In $D$, we add a first row and a first column and get the matrix $A=\left( a_{i\, j} \right) \in \Lambda_{n+2}^2$, such that $a_{1\, 1} =a_{1\, 2} =a_{2\, 1} =1$, $a_{1\, j} =0$ for $j=3,4,\ldots ,n+2$, $\alpha_{i\, 1} =0$ for $i=3,4,\ldots ,n+2$ and $a_{s+1\, t+1} =d_{s\, t}$ for $s,t\in \{ 1,2,\ldots ,n+1\}$. It is easy to see that the resulting matrix  $A\in \mathcal{M}_2 \subset \Delta_{n+2}^2$.

Therefore,  
\begin{equation}\label{M2=mu(n+1,n-1)}
\left| \mathcal{M}_2 \right| =\delta (n+1,2)
\end{equation}
for every integer $n\ge 2$.

From (\ref{partitionM1andM1}), (\ref{M1=mu(n+2,n)}) and (\ref{M2=mu(n+1,n-1)}) it follows that when $n\ge 2$
$$\delta (n+2,2) =\left| \Delta_{n+2}^2 \right| = \left| \mathcal{M}_1 \right| + \left| \mathcal{M}_2 \right| =\delta (n,2) + \delta (n+1,2)$$

This completes the proof.

\end{proof}

\begin{corollary} \cite{KYFib} From Theorem \ref{thtiorDelta_n_2} and Theorem \ref{ttrgamma_njk} it follows:
\begin{equation}\label{Fibgamma}
f_n =\gamma (n+2,n) ,
\end{equation}
where $f_n$ is the $n$-th element of the Fibonacci sequence.
\end{corollary}

\begin{remark}
Equations (\ref{Fobdelta}) and (\ref{Fibgamma}) are significantly different from each other. Thus, the result obtained in this paper differs from the result defined and proven in \cite{KYFib} concerning the same problem.
\end{remark}

\makeatletter
\renewcommand{\@biblabel}[1]{[#1]\hfill}
\makeatother

\end{document}